\documentclass[12pt]{article}
\usepackage{amsmath,amsfonts,amsthm,amssymb,mathtools,enumerate}
\usepackage{mathrsfs,graphicx,color,latexsym, tikz, calc}
\usepackage{hyperref} 
\usetikzlibrary{shadows}
\usetikzlibrary{patterns,arrows,decorations.pathreplacing}
\newtheorem{theorem}{Theorem}
\newtheorem{lemma}{Lemma}

\newtheorem{conjecture}{Conjecture}
\newtheorem{corollary}{Corollary}
\newtheorem{remark}{Remark}

\title{\bf \Large Spectral radius of graphs with given size and odd girth}

\author{
{\small Zhenzhen Lou$^{a,c}$,\ \ Lu Lu$^{b,}$\footnote{Corresponding author.
\newline{\it \hspace*{5mm}Email addresses:} xjdxlzz@163.com (Z. Lou), lulugdmath@163.com (L. Lu), huangxymath@163.com (X. Huang).}\ \ Xueyi Huang$^{a}$}\\[2mm]
\footnotesize $^a$ School of Mathematics, East China University of Science and Technology, Shanghai, 200237, China\\
\footnotesize $^b$ School of Mathematics and Statistics, Central South University, Changsha, Hunan, 410083, China\\
\footnotesize $^c$ College of Mathematics and Systems Science, Xinjiang University, Xinjiang, Urumqi, 830046, China \\}

\date{ }

\begin{document}

\maketitle

\begin{abstract}
Let $\mathcal{G}(m,k)$ be the set of graphs with size $m$ and odd girth (the length of shortest odd cycle)  $k$. In this paper, we determine the graph maximizing the spectral radius among $\mathcal{G}(m,k)$ when $m$ is odd. As byproducts, we show that, there is a number $\eta(m)>\sqrt{m-k+3}$ such that every non-bipartite graph $G$ with size $m$ and spectral radius $\rho\ge \eta(m,k)$ must contain an odd cycle of length less than $k$ unless $m$ is odd and $G\cong SK_{k,m}$, which is the graph obtained by subdividing an edge $k-2$ times of the complete bipartite graph $K_{2,\frac{m-k+2}{2}}$. This result implies the main results of Zhai and Shu [Discrete Math. 345 (2022)] and settles a conjecture of Li and Peng  \cite{li-peng} as well.\\[1mm]

\noindent {\it AMS classification:} 05C50\\[1mm]
\noindent {\it Keywords}: Spectral radius; Spectral Tur\'{a}n number; Odd cycles
\end{abstract}

\baselineskip=0.202in

\section{Introduction}
Let $\mathcal{H}$ be a set of some fixed graphs. A graph is said to be $\mathcal{H}$-free
if it does not contain  subgraphs isomorphic to any members of $\mathcal{H}$. In 1907, Mantel \cite{Mantel} presented the following famous result, which aroused the study of the so-called Tur\'{a}n-Type extremal problem in graph theory.

\begin{theorem}[Mantel \cite{Mantel}]
Let $G$ be an $n$-vertex graph. If $G$ is triangle-free, then
$m(G)\le m(K_{\lfloor\frac{n}{2}\rfloor,\lceil\frac{n}{2}\rceil})
$, equality holds if and only if $G\cong K_{\lfloor\frac{n}{2}\rfloor,\lceil\frac{n}{2}\rceil}$.
\end{theorem}

Problems involving triangles play an important role in the development of both extremal and spectral extremal graph theory. 
A classic Mantel's theorem, implies that every $n$-vertex graph of 
size $m>\lfloor\frac{n^2}4\rfloor$ must contain a triangle. 
Since then, much attentions have been paid to Tur\'{a}n-Type Extremal Problem and  Mantel's Theorem 
 has many interesting applications. 
Further related results may be found dotted throughout the literature in the nice survey \cite{Furedi} for example. 
For a graph $G$, let $\rho(G)$ be the spectral radius of $G$. 
In 1970, Nosal \cite{Nosal}  proved that if a graph $G$ 
is triangle-free with $m$ edges, then  
$\rho (G)\le \sqrt{m}$. 
This a classic result in spectral graph theory for triangle-free graphs, 
and it can be viewed as 
the spectral version of Mantel’s Theorem, also
usually called the spectral Mantel theorem. 
More precisely, Nikiforov \cite{Niki2009jctb} characterized the extremal triangle-free graphs 
attaining the upper bound. 
In what follows, we state this spectral result in a complete form.

\begin{theorem}[\cite{Nosal,Niki2009jctb}] \label{thm-NN}
Let $G$ be a graph with $m$ edges. If $G$
is triangle-free, then
$\rho(G) \le \sqrt{m}$,
equality holds if and only if $G$ is a complete bipartite graph.
\end{theorem}

During the years this striking and elegant result has attracted 
greatly significant attention
(see, e.g., \cite{Niki2002cpc}, \cite{Niki2007laa2}, \cite{Lin-Ning-Wu}, \cite{zhai-lin-shu}, \cite{Nikiforov}, \cite{LP2022} and the  survey \cite{NikifSurvey} for some highlights). In 2021, Lin, Ning and Wu \cite{Lin-Ning-Wu} obtained a new spectral condition for triangles by majorization theory. 
Consequently, they enhanced Theorem \ref{thm-NN} 
for non-bipartite graphs.

\begin{theorem}[Lin, Ning and Wu \cite{Lin-Ning-Wu}]\label{Lin}
Let $G$ be a non-bipartite graph with size $m$. If $\rho(G)\geq \sqrt{m-1}$, then $G$ contains a triangle unless $G$ is a $C_5$.
\end{theorem}

Theorem \ref{Lin} can be viewed as a stability result on 
Theorem \ref{thm-NN}, since it excluded the complete bipartite graphs 
in our consideration. 
This type of result is also regarded as the second extremal graph problem in the literature. 
Moreover, 
by a well-known inequality $\rho(G)\geq \frac {2m(G)}{n}$ originated from Collatz and Sinogowitz \cite{Collatz},
one have $\rho(G)>\sqrt{m-1}$, 
which  provided by $m(G)>\lfloor\frac{n^2}{4}\rfloor$.
Therefore, Theorem \ref{Lin} is slightly stronger than Mantel's theorem. 
Note that the extremal graph in Theorem \ref{Lin} 
is attained only for $m=5$ and $G=C_5$. For $k\ge 1$, let $S_{2k-1}(K_{2,\frac{m-2k+1}{2}})$ 
be the graph obtained from $K_{2,\frac{m-2k+1}{2}}$ 
by putting $2k-1$ new vertices on an edge. 
Clearly, $S_{2k-1}(K_{2,\frac{m-2k+1}{2}})$ is non-bipartite and 
it also is $\{C_3,C_5,\ldots ,C_{2k+1}\}$-free. In 2022, Zhai and Shu \cite{zhai}
proved a further improvement on Theorem \ref{Lin}. 
Let $\rho^*(m)$ be the largest root of $x^5-mx^3+(2m-5)x-(m-3)=0$. 

\begin{theorem}[Zhai and Shu \cite{zhai}]\label{thm-zhai-shu}
 Let $G$ be a non-bipartite graph of size $m$. If
 $\rho(G)\geq \rho^*(m)$, then $G$ contains a triangle unless 
 $m$ is odd and $G\cong S_1(K_{2,\frac{m-1}{2}})$.
  \end{theorem}

Notice that for $m\geq 6$,
we have $ \sqrt{m-2} < \rho^*(m)< \sqrt{m-1}$. Theorem \ref{thm-zhai-shu}  implies both Mantel's theorem and Lin, Ning and Wu's result as well. 
In 2022, Wang \cite[Theorem 5]{Wang2022DM} improved slightly 
Theorem \ref{thm-zhai-shu} 
by determining the $m$-edge graphs $G$ for every $m$, 
if $G$ is a  triangle-free and non-bipartite graph with 
$\rho (G) \ge \sqrt{m-2}$.   

Very recently, 
by applying Cauchy’s interlacing theorem of all eigenvalues, Li and Peng \cite{li-peng} found some forbidden induced subgraphs and  presented an alternative proof of Theorem \ref {thm-zhai-shu}. 
Note that the unique extremal graph in  Theorem \ref {thm-zhai-shu} 
contains many copies of $C_5$. 
Moreover, Li and Peng \cite{li-peng} considered the further stability result on Theorem \ref {thm-zhai-shu} by forbidding 
both $C_3$ and $C_5$ as below. 
Let  $\gamma(m)$ 
denote the largest root of 
$x^7-mx^5+(4m-14)x^3-(3m-14)x-m+5$. 

\begin{theorem}[Li and Peng \cite{li-peng}]\label{li-peng}
Let $G$ be a graph with $m$ edges. If $G$ is
$\{C_3,C_ 5\}$-free and $G$ is non-bipartite, then
$\rho(G)\leq \gamma(m)$,
equality holds if and only if $m$ is odd and $G\cong S_3(K_{2,\frac{m-3}{2}})$. 
\end{theorem}

Moreover, Li and Peng \cite[Conjecture 4.1]{li-peng} proposed the following conjecture. 
\begin{conjecture}[Li and Peng \cite{li-peng}]\label{conj-1}
Let $G$ be a graph with $m$ edges. If $G$ does not contain any member of $\{C_3,C_5,\ldots, C_{2k+1}\}$ and $G$ is non-bipartite, then
$\rho(G)\leq\rho(S_{2k-1}(K_{2,\frac{m-2k+1}{2}}))$,
equality holds if and only if $m$ is odd and $G\cong S_{2k-1}(K_{2,\frac{m-2k+1}{2}})$.
\end{conjecture}

It is worth mentioning that the analogous result of  
Conjecture \ref{conj-1} for $\{C_3,C_5,\ldots ,C_{2k+1}\}$-free 
non-bipartite graphs with given order $n$ 
was previously proposed in \cite{Lin-Ning-Wu}. 
The problem was proved by Lin and Guo \cite{LG2021} and also independently proved by 
Li, Sun and Yu \cite[Theorem 1.6]{S.Li} using a different method. 

In this paper, we aim to generalize all the above results 
and confirm Conjecture \ref{conj-1}. 
Unlike the techniques in \cite{li-peng} where the authors developed the key ideas from \cite{Lin-Ning-Wu}, in present paper, 
we will expand the ideas mainly from \cite{zhai}. 
The odd girth of a graph is defined as the length of a shortest odd cycle. 
Recall that $\mathcal{G}(m,k)$ be the set of graphs with size $m$ and odd girth   $k$. 
Let $C_k(a,b)$ be the graph obtained from a cycle $C_k=v_1v_2\cdots v_k v_1$ by replacing the edge $v_1v_k$ with a complete bipartite graph $K_{a,b}$ (see Fig.\ref{fig-1}). 
On the other hand, 
$C_k(a,b)$ is just a local blow up of the cycle $C_k$, and 
it also can be seen as a subdivision of $K_{a+1,b+1}$ 
by subdividing an edge with a path of length $k-3$. 
In particular, 
we denote $SK_{k,m}=C_k(1,\frac{m-k+2}{2})$, 
where $m\ge k\geq 5$ is an odd integer; see Fig.\ref{fig-1}. 
Clearly, $S_1(K_{2,\frac{m-1}{2}})$ in Theorem \ref{thm-zhai-shu} is just $SK_{5,m}$, and
$S_3(K_{2,\frac{m-3}{2}})$ in Theorem 
\ref{li-peng} is just $SK_{7,m}$.  

\begin{figure}[htbp]
    \centering
    \includegraphics[width=12cm]{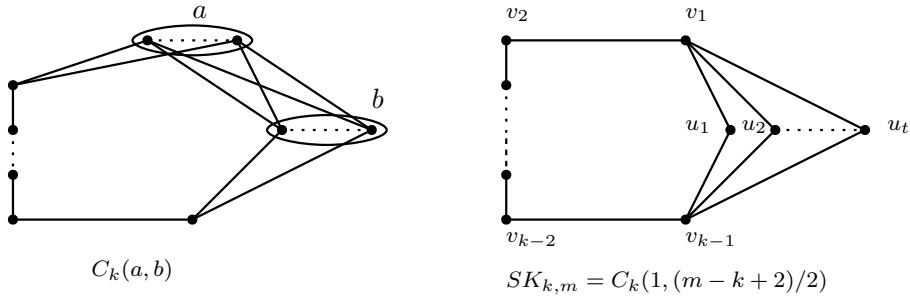}
    \caption{The graph $C_k(a,b)$ and  $SK_{k,m}$.}
    \label{fig-1}
\end{figure}

Now we present the main result in this paper.

\begin{theorem}\label{thm-1.1}
For two odd integers $m\ge k\ge 5$, if $G\in\mathcal{G}(m,k)$, then 
\[ \rho(G)\le \rho(SK_{k,m}), \] 
 with equality if and only if  $G\cong SK_{k,m}$.
\end{theorem}

Note that, all lemmas in our proof (see Section \ref{sec-3}) only use the assumption $\rho\ge (m-k+2)/\sqrt{m-k+1}>\sqrt{m-k+3}$ and $\rho$ is as large as possible. Therefore, if we assume that $G$ is a non-bipartite graph of size $m$ and $\rho(G)\ge (m-k+2)/\sqrt{m-k+1}>\sqrt{m-k+3}$ with $\rho(G)$ as large as possible, then we get $G=C_k(a,b)$ or $SK_{k,m}$. 
According to the proofs in Section \ref{sec-3}, we actually get the following stronger result.

\begin{theorem}\label{thm-1.2}
Let $G$ be a non-bipartite graph with size $m$, and $k\ge 5$ an odd integer. If $\rho(G)\ge\frac{m-k+2}{\sqrt{m-k+1}}>\sqrt{m-k+3}$, then $G$ contains an odd cycle of length less than $k$ unless $G=C_5(2,2)$, or $m$ is odd and $SK_{k,m}$.
\end{theorem}

\begin{remark}
On the one hand, from the proofs in Section \ref{sec-3}, we actually obtain all non-bipartite graphs of size $m$ satisfying $\rho>\sqrt{m-k+3}$ for any odd integer $k\ge 5$. On the other hand, though Theorem \ref{thm-1.2} removes the assumption of $m$ being odd, it does not means that the extremal graph maximizing $\rho$ is obtained among $\mathcal{G}(m,k)$ when $m$ is even. If $m=k+5$ and $k=5$, Theorem \ref{thm-1.2} indicates that the extremal graph is exactly $C_5(2,2)$. For other cases, Theorem \ref{thm-1.2} only implies $\rho(G)< \frac{m-k+2}{\sqrt{m-k+1}}$ for any $G\in\mathcal{G}(m,k)$ when $m$ is even. The extremal graph is still unknown in general. At last, we believe that our method is also valid when $m$ is even. In fact, when $m$ is even, it is easy to see that $\rho>\sqrt{m-k+2}$. By replacing $\sqrt{m-k+3}$ with $\sqrt{m-k+2}$ in Section \ref{sec-3}, the case where $m$ is even could be settled with some additional analysis.
\end{remark}

By the knowledge of equitable partition {\cite[Page 198]{Godsil}}, $\rho(SK_{k,m})$ is the largest root of $p(x)$, where $p(x)=\det(xI-B)$ and $B$ is the $k\times k$ matrix given as
\[B=\begin{pmatrix}0& 1 &0&\cdots&0&0&1\\
\frac{m-k+2}{2}&0&1&\cdots&0&0&0\\
\vdots&\vdots&\vdots& &\vdots&\vdots&\vdots\\
0&0&0&\cdots&1&0&1\\
\frac{m-k+2}{2}&0&0&\cdots&0&1&0
\end{pmatrix}.\]
Denote by $\eta(m,k)$ the largest root of $p(x)$. Similar to Lemma \ref{2-lem-1}, it is easy to see $\eta(m,k)\ge\frac{m-k+2}{\sqrt{m-k+1}}>\sqrt{m-k+3}$. Note that $\rho(C_5(2,2))=2.9135<2.9191=\eta(10,5)$, Theorem \ref{thm-1.2} implies the following result.

\begin{corollary}\label{cor-1}
Let $G$ be a non-bipartite graph with size $m$, and $k\ge 5$ an odd integer. If $\rho(G)\ge\eta(m,k)$, then $G$ contains an odd cycle of length less than $k$ unless $G\cong SK_{k,m}$.
\end{corollary} 
Clearly, Theorem \ref{thm-zhai-shu} is the special case for $k=5$ of Corollary \ref{cor-1}, and Theorem \ref{li-peng} is the special case for $k=7$. Furthermore, Corollary \ref{cor-1} completely solves Conjecture \ref{conj-1}. 

\section{Preliminaries}\label{sec-2}
At the beginning of this section, we shall give some terminologies and notations.
For two disjoint subsets $S,T\subseteq V(G)$, $e(S)$ is the number of edges with both endpoints in $S$ and
$e(S,T)$ is the number of edges with one endpoint in $S$ and the other in $T$.
Given a vertex $u\in V(G)$, $N(u)$ is the neighborhood of $u$ in $G$, and for $k\geq2$,
$N^k(u)$ is the set of vertices of distance $k$ to $u$. Denote by $d_G(u)$ the degree of $u$ and
$d_S(u)$ the number of its neighbors in $S\subseteq V(G)$. Let $G[S]$ be the subgraph of $G$ induced by $S$.
Let $\phi(G)$ be the characteristic polynomial of $A(G)$. The spectral radius $\rho(G)$ is the largest root of $\phi(G)$.
Its corresponding unit eigenvector is called the {\it Perron vector} of $G$.
From the famous Perron-Frobenius theorem, Perron vector is a positive vector for a connected graph $G$.
\begin{figure}[htbp]
    \centering
    \includegraphics[width=12cm]{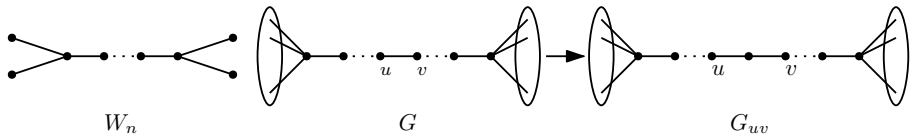}
    \caption{\footnotesize The graph $W_n$ and the subdivision operation.}
    \label{fig-0}
\end{figure}

An \emph{internal path} of $G$ is a sequence of  vertices $v_1,v_2,...,v_s$ with $s\ge 2$ such that:\\
(i) the vertices in the sequence are distinct (except possibly $v_1=v_s$),\\
(ii) $v_i$ is adjacent to $v_{i+1}$ ($i=1,2,...,s-1$),\\
(iii) the vertex degrees satisfy $d(v_1)\ge3$, $d(v_2)=\cdots=d(v_{s-1})=2$ (unless $s=2$) and $d(v_s)\ge3$.
\begin{lemma}[\cite{Hoffman}]\label{sub-inn}
Suppose that $G\not= W_n$ (see Fig.\ref{fig-0}) and $uv$ is an edge on  an internal path of $G$. Let $G_{uv}$ be the graph obtained from $G$ by the subdivision of  the edge $uv$. Then $\rho(G_{uv})<\rho(G)$.
\end{lemma}

Now we give upper and lower bounds for the spectral radius of $SK_{k,m}$.

\begin{lemma} 
\label{2-lem-1}
For every odd integers $m\ge k\ge 5$, we have  
\[ \sqrt{m-k+3}<  \frac{m-k+2}{\sqrt{m-k+1}} \le \rho(SK_{k,m})
<\sqrt{m-k+4} .\] 
\end{lemma}
\begin{proof}
For convenience, denote by $\rho=\rho(SK_{k,m})$ and $t=\frac{m-k+2}{2}$. It suffices to show $\frac{2t}{\sqrt{2t-1}}\le \rho<\sqrt{2t+2}$. By Lemma \ref{sub-inn}, we have $\rho<\rho(K_{2,t+1})=\sqrt{2t+2}$. 

For any positive number $x$, denote by $\alpha_x=\frac{x+\sqrt{x^2-4}}{2}$ and $\beta_x=\frac{x-\sqrt{x^2-4}}{2}$. Therefore, one can easily verify that $\rho\ge\frac{2t}{\sqrt{2t-1}}$ if and only if $\frac{\rho}{2t\beta_{\rho}}\ge 1$. In what follows, we show $\frac{\rho}{2t\beta_{\rho}}\ge 1$. 

For any odd integer $r\ge 5$, let $G_r=SK_{r,2t+r-2}$ and $\rho_r=\rho(G_r)$. We have proved that $\rho_r<\sqrt{2t+2}$. For each $G_r$, label $G_r$ like $SK_{k,m}$ in Fig.\ref{fig-1}, that is, $v_1$ and $v_{r-1}$ are the only vertices of degree greater than $2$, $v_1v_2\cdots v_{r-1}$ is the path, and $N(v_1)\setminus\{v_2\}=N(v_{r-1})\setminus\{v_{k-2}\}=\{u_1,\ldots,u_t\}$. Let $f$ be the eigenvector such that $A(G_r)f=\rho_rf$ and $f(u_1)=f(u_2)=\cdots=f(u_t)=2$. By $A(G_r)f=\rho_rf$, we have
\[\left\{\begin{array}{l}
2\rho_r=f(v_1)+f(v_{k-1})=2f(v_1),\\[2mm]
\rho_r f(v_1)=2t+f(v_2),\\[2mm]
\rho_r f(v_i)=f(v_{i-1})+f(v_{i+1}), \mbox{ for }3\le i\le \ell,
\end{array}
\right.
\]
where $\ell=\frac{r-1}{2}$. By immediate calculations, we have $f(v_i)=\frac{\rho_r-2t\beta_{\rho_r}}{\alpha_{\rho_r}-\beta_{\rho_r}}\alpha_{\rho_r}^i+\frac{2t\alpha_{\rho_r}-\rho}{\alpha_{\rho_r}-\beta_{\rho_r}}\beta_{\rho_{r}}^i$ for $1\le i\le \ell$.
Since $f(v_{\ell})>0$, by simple calculations, we have
\begin{equation}\label{eq-lem2}
    \frac{\rho_r}{2t\beta_{\rho_r}}>\frac{1-(\beta_{\rho_r}/\alpha_{\rho_r})^{\ell-1}}{1-(\beta_{\rho_r}/\alpha_{\rho_r})^{\ell-1}}>\frac{1-c_t^{\ell-1}}{1-c_t^{\ell}},
\end{equation}
where $c_t=\frac{\sqrt{2t+2}-\sqrt{2t-2}}{\sqrt{2t+2}+\sqrt{2t-2}}$, and the last inequality is due to the fact that $\frac{1-x^{\ell-1}}{1-x^{\ell}}$ is a decrease function when $0<x<1$, $\frac{\beta_{x}}{\alpha_{x}}$ is a increase function, and $\rho_r<\sqrt{2t+2}$. Keep in mind that the function 
\[g(x)=\frac{x}{2t\beta_{x}}=\frac{x}{t(x-\sqrt{x^2-4})}\]
increases along with $x$ increasing when $x>2$.

Now suppose to the contrary that $\frac{\rho}{2t\beta_{\rho}}<1$. There exists $\epsilon>0$ such that $\frac{\rho}{2t\beta_{\rho}}<1-\epsilon$. Since $0<c_t<1$, there exists $N>0$ such that $\frac{1-c_t^{n-1}}{1-c_t^{n}}>1-\epsilon$ for any $n\ge N$. Take $r=\max\{k,2N+1\}$, and $\ell=\frac{r-1}{2}$. Therefore, $\rho\ge\rho_r$ due to Lemma \ref{sub-inn}, and thus
\[\frac{\rho}{2t\beta_{\rho}}\ge \frac{\rho_r}{2t\beta_{\rho_r}}>\frac{1-c_t^{\ell-1}}{1-c_t^{\ell}}>1-\epsilon,\]
a contradiction.
\end{proof}

\begin{lemma}[\cite{Wu}]\label{lem-2.4}
Let $u$, $v$ be two distinct vertices in a connected graph $G$, and $G'=G-\{vv_i\mid 1\leq i\leq s\}+\{uv_i\mid 1\leq i\leq s\}$, where $\{v_i\mid i=1,2,\ldots, s\}\subseteq N_G(v)\setminus
N_G(u)$. Assume that $x$ is the Perron vector of $G$. If $x_u\geq x_v$, then $\rho(G)<\rho(G')$.
\end{lemma}

\begin{lemma}[\cite{Cvetkovi}]
\label{lem-2.5}
Let $s$, $t$, $u$, $v$ be the four distinct vertices of a connected graph $G$ and let $st, uv\in E(G)$ , while $sv, tu \notin  E(G)$. 
If $(x_s-x_u)(x_v-x_t)\geq 0$, where $x$ is the Perron vector of $G$, then 
$\rho(G-st-uv+sv+tu) \geq \rho(G)$ with equality if and only if $x_s= x_u$ and $x_t = x_v$.
\end{lemma}

Let $H$ be a subgraph of $G$. For $u,v\in V(H)$, denote by $d_H(u,v)$ the distance of $u,v$ in $H$, that is, the length of a shortest path from $u$ to $v$ in $H$. We close this part by the following result, which is immediate by simple observations and its proof is omitted. 
\begin{lemma}\label{cyc}
Let $C$ be an odd cycle of length $l\ge 5$ in graph $G$. For two vertices $x,y\in V(C)$ and $u\in V(G)\setminus V(C)$ with $x\sim u$ and $y\sim u$, if $d_C(x,y)\ge 2$, then $G$ contains an odd cycle containing $u$ with length at most $l$. Furthermore, if $d_C(x,y)\ge 3$, then such cycle has length at most $l-1$. 
\end{lemma}
\

\section{Proofs.}\label{sec-3}
In this part, we give the proof of Theorem \ref{thm-1.1}. 
The techniques in our proof are motivated by that of Zhai and Shu \cite{zhai}. 
Avoiding fussy repetition and tedious calculations, we always assume that $k\ge 7$ though our method is valid for $k=5$. Let $G^*$ be the extremal graph with the maximum spectral radius among all graphs in $\mathcal{G}(m,k)$. By Lemma \ref{2-lem-1}, the lower bound of $G^*$ is given by
\begin{equation}\label{lem-3.1}
  \rho^2(G^*)\geq \rho^2(SK_{k,m})\ge\frac{(m-k+2)^2}{m-k+1}>m-k+3.
\end{equation}
Our goal is to determine the structure of $G^*$. Let $C_k$ be a shortest cycle of $G^*$. Keep in mind that $G^*$ has no odd cycle with length less than $k$. We always assume that ${\bf x}=(x_1,x_2,\ldots,x_{|V(G)|})$ is the Perron vector of $G^*$, and $u^*$ is a vertex with the largest component in ${\bf x}$, i.e., $x_{u^*}=\max\{x_i\mid i\in V(G^*)\}$.
Denoted by $A=N(u^*), B=V(G^*)\setminus{(A \cup {u^*})}$ and $\rho^*=\rho(G^*)$. Note that $e(A)=0$ because $G^*$ contains no triangle. This fact yields an upper bound of $e(B)$.
\begin{lemma}\label{lem-3.3}
 $e(B)\leq k-4$. 
\end{lemma}
\begin{proof}
Since $e(A)=0$, we have $\sum_{i\in A}d_A(i)x_i=0$. 
Note that 
\[ {\rho^{*}}^{2}x_{u^*}=\rho^*(\rho^* x_{u^*})
=\rho^*\left(\sum\limits_{v\in A}x_v \right)=\sum\limits_{v\in A} 
\rho^*x_v=\sum\limits_{v \in A}{\sum\limits_{u\in N(v)}{x_u}}. \]
According to the assumption \eqref{lem-3.1}, we have
\begin{align*}
(m-k+3)x_{u^*}&<d(u^*)x_{u^*}+\sum\limits_{i\in A}d_A(i)x_i+\sum\limits_{j\in B}d_A(j)x_j=d(u^*)x_{u^*}+\sum\limits_{j\in B}d_A(j)x_j\\
&\leq\left(d(u^*)+e(A, B)\right)x_{u^*}=\left( m-e(B)\right)x_{u^*}.  
\end{align*} 
It leads to $e(B)<k-3$, and thus $e(B)\leq k-4$.
\end{proof}

\begin{lemma}\label{lem-3.4}
$u^*$ must be on a shortest odd cycle of $G^*$.
\end{lemma}
\begin{proof}
Let $C$ be a shortest odd cycle of $G^*$. If $u^*\in C$, there is nothing to prove. Otherwise, since $e(A)=0$ and $e(B)\leq k-4$, there exists at least two vertices $x,y\in A\cap V(C)$ with $d_C(x,y)\ge 2$. Hence we will get a cycle containing $u^*$ with length at most $k$ by Lemma \ref{cyc}.
\end{proof}

\begin{lemma}\label{lem-3.5}
$B$ induces a $P_{k-3}$ with some isolates possibly.
\end{lemma}

\begin{proof}
According to Lemma \ref{lem-3.4}, let 
$C=u^*u_0v_1v_2\ldots v_{k-3}u_au^*$
be a shortest odd cycle of $G^*$, where $u_0,u_{a}\in A$. We claim that $v_1,v_2,\ldots,v_{k-3}\in B$ since otherwise there would be a shorter odd cycle in $G^*$. Note that $e(B)\leq k-4$. The result follows.
\end{proof}

\begin{figure}[htbp]
    \centering
    \includegraphics[width=16cm]{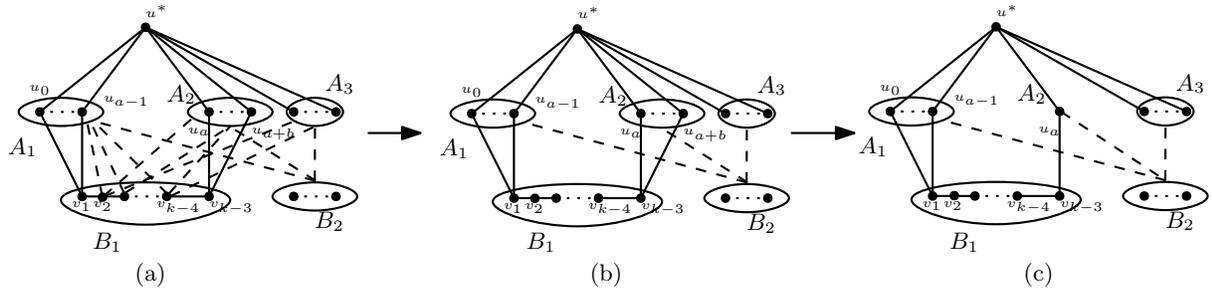}
    \caption{\footnotesize The structures yielded by Lemmas \ref{lem-3.5}, \ref{lem-3.7} and \ref{lem-3.8} where the dashed lines between two parts means that the adjacency relations between the vertices in one part and those in the other part are still unknown.}
    \label{fig-3}
\end{figure}

Let $B_1=V(P_{k-3}) =\{v_1,v_2,\ldots v_{k-3}\}$. 
Then $B_2=B\setminus B_1$ is an independent set in $G$. Note that $N_{A}(v_1)\cap N_{A}(v_{k-3})=\emptyset$ since otherwise $G^*$ would contain a shorter odd cycle. 
So, we may denote $A_1=N_{A}(v_1)=\{u_0,u_1,\ldots,u_{a-1}\}$ for some $a\geq 1$ and $A_2=N_{A}(v_{k-3})=\{u_a,u_{a+1},\ldots,u_{a+b}\}$ for some $ b\geq 0$. Set $A_3=A\setminus (A_1 \cup A_2)$. Clearly, $A_1,A_2$ and $A_3$ are  independent sets. Now the structure of $G^*$ is as shown in Fig.\ref{fig-3} (a).

\begin{lemma}\label{lem-3.7}
$d(v_i)=2$ for each $i=2,3,\ldots ,k-4$.
\end{lemma}
\begin{proof}
Suppose to the contrary that $ v_i\sim u$ for some $u\in A$  and $2\le i\le k-4$. Consider the cycle $C=u^*u_0v_1\cdots v_{k-3}u_au^*$. If $u\in \{u_0,u_a\}$, clearly, then there would an  odd cycle of length 
shorter than $k$. 
Indeed, if $v_i \sim u_0$, then either the cycle 
$u_0v_1v_2\cdots v_i u_0$ or the cycle $u^*u_0v_iv_{i+1}\cdots v_{k-3}u_au^*$ 
is an odd cycle with shorter length. 
This is a contradiction. If $u\in  A\setminus\{u_0,u_a\}$, noticing $d_C(u^*,v_i)\ge 3$ and $u^*\sim u$ and $v_i\sim u$, Lemma \ref{cyc} indicates there would be a shorter odd cycle as well, a contradiction.
\end{proof}

According to Lemma \ref{lem-3.7}, the structure of $G^*$ is as shown in Fig.\ref{fig-3} (b). Without loss of generality, we may assume that $x_{v_{1}}\ge x_{v_{k-3}}$ in what follows.

\begin{lemma}\label{lem-3.8}
$|A_2|=1$, i.e., $A_2=\{u_a\}$.
\end{lemma}
\begin{proof}
Suppose to the contrary that $|A_2|\geq2$, that is,  $A_2=\{u_a,\ldots, u_{a+b}\}$ with $b\ge 1$.
Let $G'=G^*-\{u_iv_{k-3}\mid i=1,2,\ldots,b\}+\{u_iv_1\mid i=1,2,\ldots,b\}$. Clearly, $G'\in\mathcal{G}(m,k)$ and  Lemma \ref{lem-2.4} implies $\rho(G')>\rho(G^*)$, a contradiction.
\end{proof}

\begin{lemma}\label{lem-3.8'}
If  $B_2= \emptyset$, then $G^*\cong SK_{k,m}$.
\end{lemma}
\begin{proof}
If $c=|A_3|=0$, we have $G^*\cong SK_{k,m}$, it follows the result.
In the following, we suppose that  $c=|A_3|\geq 1$.
At this time, $G^*$ is isomorphic to $SK_{k,m-c}$ by attaching $c$ pendent edges at its  maximal degree vertex.
Notice that $m=2a+c+k-2$. Let $G'$ obtained from $G^*$ by contracting the internal path $P_{k-3}$ as an edge $v_1v_2$, that is $G'$ is obtained from $SK_{5,m-c}$ by attaching $c$ pendent edges to its  maximal degree vertex. 
By Lemma \ref{sub-inn}, we have $\rho(G^*)<\rho(G')$.
 Next we will show $\rho(G')\le \frac{m-k+2}{\sqrt{m-k+1}}=\frac{2a+c}{\sqrt{2a+c-1}}$, which yields a contradiction.
By the knowledge of equitable partition \cite[Page 198]{Godsil}, $\rho(G')$ is the largest root of 
$f(x)=x^6-(2a+c+3)x^4+(4a+2c+ac+1)x^2-2ax-ac$.
The derivative function of $f(x)$ is
$f'(x)=6x^5-4(2a+c+3)x^3+2(4a+2c+ac+1)x-2a$. One can verify that 
$f'(x)>0$ for $x>\sqrt{2a+c+1}$. Thus $f(x)$ is increase when $x\geq \sqrt{2a+c+1}$.
By computation, 
\begin{align*}
f(\sqrt{2a+c+1})=&(ac-1)(2a+c+1)-2a\sqrt{2a+c+1}-ac.
\end{align*}

If $a=1$, $c\geq 3$,  then  $f(\sqrt{2a+c+1})=f(\sqrt{c+3})=(c+3)^2-5(c+3)-2\sqrt{c+3}+3$ and we can verify it is an increase function about $c$, and thus
 $f(\sqrt{c+3})\geq f(\sqrt{3+3})=7-\frac{1}{\sqrt{6}}>0$. Thus, $\rho(G')<\sqrt{2a+c+1}<\frac{2a+c}{\sqrt{2a+c-1}}$.
 
If $a\geq 2$ and $c\geq 2$, we have
\begin{align*}
f(\sqrt{2a+c+1})=&(ac-1)(2a+c+1)-2a\sqrt{2a+c+1}-ac\\
=&ac[(1-\frac{1}{ac})(2a+c+1)-\frac{2}{c}\sqrt{2a+c+1}-1]\\
\geq&ac[(1-\frac{1}{4})(2a+c+1)-\sqrt{2a+c+1}-1]\\
=&4ac[3(2a+c+1)-4\sqrt{2a+c+1}-4]\\
>&0,
\end{align*}
due to $\sqrt{m-k+3}=\sqrt{2a+c+1}\geq\sqrt{7}>2$. Hence, $\rho(G')<\sqrt{2a+c+1}<\frac{2a+c}{\sqrt{2a+c-1}}$.

If $a=1$ and $c=2$, we have
$f(\frac{2a+c}{\sqrt{2a+c-1}})=f(\frac{\sqrt{3}}{4})=4.6405>0$.
It follows that $\rho(G')\le\frac{\sqrt{3}}{4}= \frac{2a+c}{\sqrt{2a+c-1}}$.

If $a=1$ and $c=1$, then $\rho^*\le \rho(G')=2.115<\frac{3}{\sqrt{2}}=\frac{2a+c}{\sqrt{2a+c-1}}$.
\end{proof}

According to Lemma \ref{lem-3.8}, if $B_2\neq \emptyset$, the structure of $G^*$ is as shown in Fig.\ref{fig-3} (c). Now we turn our eyes on the components of the Perron vector $\mathbf{x}$. Without loss of generality, we may assume that $x_{u_0}\ge x_{u_{1}}\ge\cdots\ge x_{u_{a-1}}$ 
and $B_2\neq \emptyset$ in the rest proof.

\begin{lemma}\label{lem-3.10}
If $B_2\neq \emptyset$, then for any $w\in B_2$, let $s=d_A(w)$, we have $s\geq 2$ and the value
\begin{equation}\label{eq-3.2}
    M(w)=a(x_{v_1}-x_{u^*})+s(x_{w}-x_{u^*})+x_{v_{k-3}}> 0.
\end{equation}
\end{lemma}

\begin{proof}
If $s=1$ and $u_i\sim w$, then $d_G(w)=d_A(w)=1$ 
since $B_2$ is an independent set. 
The graph $G'=G^*-u_iw+ u^*w\in \mathcal{G}(m,k)$ and $\rho(G')>\rho(G^*)$ due to Lemma \ref{lem-2.4}, which is impossible. Hence $d_A(w)=s\geq 2$.
 
Let $N^2(u^*)$ be the set of vertices $v
\in V(G)$ such that there exists 
$u\in V(G)$ satisfying $u^*\sim u \sim v$. 
Trivially, we have $u^*\in N^2(u^*)$. 
According to $A(G^*)\mathbf{x}=\rho^*\mathbf{x}$, we get  
\begin{align*}
(\rho^*)^2 x_{u^*}=&(A^2x)_{u^*}=
\sum\limits_{v\in N^2(u^*)}d_{A}(v)x_v\\
=& d_{A}(v_1) x_{v_1}+d_{A}(v_{k-3}) x_{v_{k-3}}+d_{A}(w) x_w+
\sum\limits_{v\in N^2(u^*)\setminus\{v_1,v_{k-3},w\}}d_{A}(v)x_v\\ 
\leq& ax_{v_1}+x_{v_{k-3}}+sx_{w}+(m-(k-4)-a-1-s)x_{u^*}\\
=&(m-k+3)x_{u^*}+a(x_{v_1}-x_{u^*})+s(x_{w}-x_{u^*})+x_{v_{k-3}}\\
=&(m-k+3)x_{u^*}+M(w).
\end{align*}
Since $(\rho^*)^2>(m-k+3)$, we get $M(w)>0$.
\end{proof}

For components of $\bf{x}$ corresponding to $A$, we get the following result. 

\begin{lemma}\label{lem-3.9}
$x_{u_{a-1}}\geq x_{u_a}$ and if $A_3\neq \emptyset$, then $x_{u_a}\ge x_w$ for any $w\in A_3$.
\end{lemma}

\begin{proof}
Suppose to the contrary that $x_{u_{a-1}}<x_{u_a}$. Recall that $x_{v_1}\geq x_{v_{k-3}}$.
Let 
$$G'=G^*-u_{a-1}v_1-u_av_{k-3}+u_{a-1}v_{k-3}+u_av_1.$$
Clearly, $G'\in\mathcal{G}(m,k)$ and Lemma \ref{lem-2.5} implies 
$\rho(G')>\rho(G)$, a contradiction. Suppose to the contrary that
$x_{a}\leq x_w$ for some $w\in A_3$. Let $G''=G^*-u_{a-1}v_1+u_av_1$.
Similarly, $G''\in\mathcal{G}(m,k)$ and 
$\rho(G'')>\rho(G)$, a contradiction.
\end{proof}

 According to Lemma \ref{lem-3.9}, we know that 
 $x_{u_{a-1}}\geq x_{u_a}$ and $x_{u_a}\ge x_w$ for any $w\in A_3$. 
We denote $A_3=\{u_{a+1},u_{a+2},\ldots, u_{a+c}\}$ 
for some $c\ge 0$. 
By  
sorting the vertices of $A_1$ and $A_3$, we may assume further that 
\begin{equation} \label{eq-order}
x_{u_0}\ge\cdots \ge x_{u_{a-1}}\ge x_{u_a}\ge x_{u_{a+1}}\ge\cdots\ge x_{u_{a+c}}.
\end{equation} 

\begin{lemma}\label{lem-3.11}
If $B_2\neq \emptyset$, then for $w\in B_2$, let $d(w)=s$, we have  $N(w)=\{u_0,u_1,\ldots,u_{s-1}\}$.
\end{lemma}
\begin{proof}
Suppose to the contrary that $u_{i}\sim w$ with $i\ge s$. Therefore, there exists $1\le j\le s-1$ such that $u_j\not\sim w$. Thus, $G'=G^*-u_iw+u_jw\in\mathcal{G}(m,k)$ with $\rho(G')>\rho^*$ due to (\ref{eq-order}) and Lemma \ref{lem-2.4}, this is a contradiction.
\end{proof}

The following two lemmas (Lemmas \ref{lem-3.14} 
and \ref{lem-two}) are  key ingredients in the proof of our main result. 
To some extent, it characterized clearly the structure of the desired extremal graph. 
More precisely, we will show that if $B_2\neq \emptyset$, 
then $A_3=\emptyset$ and every vertex of $B_2$ 
is adjacent to every vertex of $A_1\cup A_2$. 
We now begin the details in earnest.

\begin{lemma}\label{lem-3.14}
If $B_2\neq \emptyset$, then $A_3=\emptyset$. 
\end{lemma}

\begin{proof}
In what follows, we will prove that $A_3= \emptyset$ 
whenever $B_2\neq \emptyset$. 
Otherwise, we assume to the contrary that
 $A_3=\{u_{a+1},u_{a+2},\ldots, u_{a+c}\}$, 
where $c=|A_3|\ge 1$.   
We will show that $\rho(G^*)<\rho(SK_{k,m})$, which leads to a contradiction. 
Since $B_2\neq \emptyset$, choosing a vertex $w\in B_2$, we compute the values $\rho^*M(w)$ and $(\rho^*)^2M(w)$ by using the equation $A(G^*)\mathbf{x}=\rho^*\mathbf{x}$. Note that Lemma \ref{lem-3.11} gives 
\[ \rho^*x_w = \sum_{i=0}^{s-1}x_{u_i}= \rho^*x_{u^*} -\sum_{j=s}^{a+c} 
x_{u_j}. \]
Moreover, we have 
\[  \begin{cases}
\rho^* x_{v_1}=\sum_{i=0}^{a-1} x_{u_i} + x_{v_2}, \\
\rho^*x_{v_{k-3}}=x_{v_{k-4}} + x_{u_a}. 
\end{cases} \]
By immediate calculations, we have
\begin{equation}\label{eq-f-1} 
\begin{aligned}
    \rho^*M(w)&= a\rho^* x_{v_1} +s\rho^*x_w - (a+s)\rho^* x_{u^*} + 
    \rho^* x_{v_{k-3}} \\ 
  &=  ax_{v_2}+x_{v_{k-4}}-(a-1)x_{u_a}-a\sum\limits_{i=a+1}^{a+c}x_{u_i}-s\sum\limits_{j=s}^{a+c}x_{u_j},
    \end{aligned}
\end{equation}
and
\begin{equation}\label{eq-f-2}
(\rho^*)^2M(w)=a(x_{v_1}+x_{v_3})+(x_{v_{k-5}}+x_{v_{k-3}})
-(a-1)\rho^* x_{u_a}-a\sum\limits_{i=a+1}^{a+c}\rho^*x_{u_i}-s\sum\limits_{j=s}^{a+c}\rho x_{u_j}.
\end{equation}
Note that $\rho^*x_{u_a}\ge x_{u^*}+x_{v_{k-3}}$ and $\rho^*x_{u_j}\ge x_{u^*}$ for $0\le j\le a+c$. 
We get from \eqref{eq-f-2} that 
\begin{equation}\label{eq-f-3} 
\begin{aligned}
    (\rho^*)^2M(w) &\le 
    ax_{v_1} +ax_{v_3} +x_{v_{k-5}} + x_{v_{k-3}} 
     - (a-1)(x_{u^*} + x_{v_{k-3}}) \\ 
     & \quad -ac x_{u^*} - s(a+c-s+1)x_{u^*} \\ 
&= ax_{v_1}+ax_{v_3}+x_{v_{k-5}}-(a-2)x_{v_{k-3}}
-f(s)x_{u^*}, 
\end{aligned}
\end{equation}
where $f(s)$ is defined as 
\[  f(s)=-s^2+(a+c+1)s+ac+a-1. \] 
Note that Lemma \ref{lem-3.10} gives $M(w)>0$, and thus $\rho^*M(w)>0 $ and $(\rho^*)^2M(w)>0$. 

{\flushleft\bf Case 1.} $a\ge 2$. 
In this case, from \eqref{eq-f-3}, we have
\begin{equation}\label{eq-f-4}
0<  (\rho^*)^2M(w)\le
ax_{v_1}+ax_{v_3}+x_{v_{k-5}}-(a-2)x_{v_{k-3}}
-f(s)x_{u^*}\le ((2a+1)-f(s))x_{u^*}.  
\end{equation}
It leads to $f(s)<2a+1$. Since $2\le s\le a+c+1$, we have 
either $3a+ac+2c-3=f(2) \le f(s)<2a+1$ 
or $ac+a-1=f(a+c+1)\le f(s)<2a+1$. Thus we get $c<1+2/a\le 2$ and so $c=1$.

Since $c=1$, we have $f(s)=-s^2+(a+2)s+2a-1$. By  $f(s)<2a+1$, we get $s=a+2$. It yields that $w\sim u$ for any $w\in B_2$ and $u\in A$. Obviously, $x_{u_0}=\cdots=x_{u_{a-1}}$ and $x_w=x_{u^*}$ for any $w\in B_2$ because $A_1$ and $B_2\cup \{x_{u^*}\}$ are orbits of $G^*$ acting by ${\rm Aut}(G^*)$. Equation \eqref{eq-f-2} turns to be 
\[(\rho^*)^2M(w)=a(x_{v_1}+x_{v_3})+(x_{v_{k-5}}+x_{v_{k-3}})
-(a-1)\rho^* x_{u_a}-a\rho^*x_{u_{a+1}}.\]
Since $\rho^*x_{u_a}=(b_2+1)x_{u^*}+x_{v_{k-3}}$ and $\rho^*x_{u_{a+1}}=(b_2+1)x_{u^*}$ where $b_2=|B_2|$, we have
\begin{align*}
(\rho^*)^2M(w)&=a(x_{v_1}+x_{v_3})+x_{v_{k-5}}-(a-2)x_{v_{k-3}}-(2a-1)(b_2+1)x_{u^*}\\[2mm]
&\le ((2a+1)-(2a-1)(b_2+1))x_{u^*}.
\end{align*}
Therefore, $(2a+1)-(2a-1)(b_2+1)>0$, and thus $b_2<\frac{2}{2a-1}\le \frac{2}{3}$. It leads to $b_2=0$, which contradicts the assumption of $B_2\ne\emptyset$.

{\flushleft\bf Case 2.} $a=1$.
In this case, from \eqref{eq-f-3}, we have
\begin{equation}\label{eq-f-5}
0<  (\rho^*)^2M(w)\le
x_{v_1}+x_{v_3}+x_{v_{k-5}}+x_{v_{k-3}}
-f(s)x_{u^*}\le (4-f(s))x_{u^*}. 
\end{equation}
It leads to $f(s)=-s^2+(c+2)s+c<4$. Since $2\le s\le c+2$, we have 
either $3c=f(2)\le f(s) <4$ or $c=f(c+2)\le f(s)<4$. Hence $c=1$, $2$ or $3$.

{\flushleft\bf Subcase 2.1.} $c=1$.

Since $2\le s\le a+c+1=3$, we have $s=2$ or $3$. Set $B_{2,1}=\{w\in B_2\mid d(w)=2\}$, $B_{2,2}=\{w\in B_2\mid d(w)=3\}$, $b_1=|B_{2,1}|$ and $b_2=|B_{2,2}|$. Similarly, $x_w=x_{u^*}$ for any $w\in B_{2,2}$. If $b_1\ge 1$, for $w\in B_{2,1}$, \eqref{eq-f-2} turns to be
\[(\rho^*)^2M(w)=x_{v_1}+x_{v_3}+x_{v_{k-5}}+x_{v_{k-3}}-3\rho^*x_{u_2}.\]
Since $\rho^*x_{u_2}=(1+b_2)x_{u^*}$, we have
\begin{align*}
    (\rho^*)^2M(w)&=x_{v_1}+x_{v_3}+x_{v_{k-5}}+x_{v_{k-3}}-3(b_2+1)x_{u^*}\\[2mm]
    &\le (4-3(b_2+1))x_{u^*}.
\end{align*}
Therefore, $3(b_2+1)<4$, and thus $b_2<1/3$. It leads to $b_2=0$. 
In this case, as similar to the proof of Lemma \ref{lem-3.8'}, we construct a new graph $G'$ by contracting the $v_1v_2\cdots v_{k-3}$ as an edge $v_1v_2$ from $G^*$. The value $\rho(G')$ is just the root of the function 
\[f(x)=x^5-x^4-(2b_1+4)x^3+(2b_1+3)x^2+(2b_1+1)-2b_1.\]
By calculations, we have $f(x)>0$ whenever $x>\sqrt{2b_1+4}$. It means that $\rho^*\le\rho(G')<\sqrt{2b_1+4}=\sqrt{m-k+3}$, a contradiction.
 Thus, $b_1=0$ and $b_2\ge 1$. For $w\in B_{2,2}$, \eqref{eq-f-2} turns to be 
\[(\rho^*)^2M(w)=x_{v_1}+x_{v_3}+x_{v_{k-5}}+x_{v_{k-3}}-\rho^*x_{u_2}.\]
Since $\rho^*x_{u_2}=(1+b_2)x_{u^*}$, we have
\begin{align*}
    (\rho^*)^2M(w)&=x_{v_1}+x_{v_3}+x_{v_{k-5}}+x_{v_{k-3}}-(b_2+1)x_{u^*}\\[2mm]
    &\le (4-(b_2+1))x_{u^*}.
\end{align*}
Therefore, $(b_2+1)<4$, and thus $b_2<3$. Let $G_1$ and $G_2$ be the graphs shown in Fig.\ref{fig-4}. According to Lemma \ref{sub-inn}, if $b_2=1$, then $\rho^*<\rho(G_1)=2.632<\sqrt{m-k+3}=\sqrt{7}$; if $b_2=2$, then $\rho^*<\rho(G_2)=3.133<\sqrt{m-k+3}=\sqrt{10}$. They are both impossible since $\rho^*>\sqrt{m-k+3}$.

\begin{figure}[htbp]
    \centering
    \includegraphics[width=8cm]{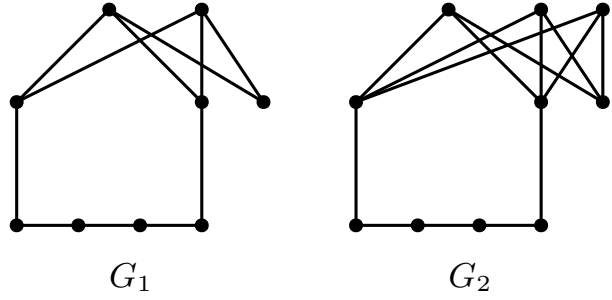}
    \caption{The graphs $G_1$ and $G_2$ in the proof of Lemma \ref{lem-3.14}.}
    \label{fig-4}
\end{figure}

{\flushleft\bf Subcase 2.2.} $c\in \{2,3\}$.

In this case, it is easy to verify that $f(s)=-s^2+(c+2)s+c<4$ implies $s=c+2$. Also, $x_w=x_{u^*}$ for any $w\in B_2$. Now \eqref{eq-f-2} turns to be 
\[(\rho^*)^2M(w)=x_{v_1}+x_{v_3}+x_{v_{k-5}}+x_{v_{k-3}}-\rho^*\sum\limits_{i=2}^{c+1}x_{u_i}.\]
Since $2\le i\le c+1$ and $\rho^*x_{u_i}=(b+1)x_{u^*}$ where $b=|B_2|$, we have
\begin{align*}
    (\rho^*)^2M(w)&=x_{v_1}+x_{v_3}+x_{v_{k-5}}+x_{v_{k-3}}-c(b+1)x_{u^*}\\[2mm]
    &\le (4-c(b+1))x_{u^*}.
\end{align*}
Therefore, $4-c(b+1)>0$, and thus $b=0$, which contradicts the assumption $B_2\ne\emptyset$.

The proof is completed.
\end{proof}

\begin{lemma} \label{lem-two} 
If $B_2\neq \emptyset$, then $N_G(w)=A_1\cup A_2$ for every $w\in B_2$. 
\end{lemma}

\begin{proof} 
Denote $s=d(w)$, we have $2\le s\le a+1$ by Lemma \ref{lem-3.10} and Lemma \ref{lem-3.14}. 
If $a=1$, then $s=2$, the result $N_G(w)=A_1\cup A_2$ 
follows immediately.
Now, we consider $a\ge 2$, 
it only needs to show that $s=a+1$. According to $A(G^*)\mathbf{x}=\rho^*\mathbf{x}$, we have $$\rho^* x_{u^*}
=\sum_{i=0}^{a-1}x_{u_i}+x_{u_a}\ \  \mbox{and}\ \ \rho^* x_{v_1}=\sum_{i=0}^{a-1}x_{u_i}+x_{v_2}.$$
Thus $x_{u_a}-x_{v_2}=\rho(x_{u*}-x_{v_1})\geq0$ and then 
 $x_{u_a}\ge x_{v_2}$.
Moreover,
$$\rho^* x_{u_a}=x_{u^*}+x_{v_{k-3}}\ \ \mbox{and}\ \ \rho^* x_{v_{k-4}}=x_{v_{k-5}}+x_{v_{k-3}}.$$
It follows that
$\rho^*(x_{u_a}-x_{v_{k-4}})=x_{u^*}-x_{k-5}\geq0$, i.e. $x_{u_a}\geq x_{v_{k-4}}$. Therefore, from \eqref{eq-f-1}, we have
\begin{align*}
\rho^*M(w)&=ax_{v_2}+x_{v_{k-4}}-
(a-1)x_{u_a}-s\sum\limits_{j=s}^{a}x_{u_j}\\
&\le ax_{v_2}+x_{v_{k-4}}-(a-1)x_{u_a}- s(a-s+1)x_{u_a}\\[2mm]
&\le (a+1-(a-1)-s(a-s+1))x_{u_a}.
\end{align*}
Note that $a+1-(a-1)-s(a-s+1)=s^2-(a+1)s+2\le 0$ whenever $2\le s\le a$, and $\rho^*M(w)>0$. We have $s=a+1$.

\end{proof}

 Recall the definition of $C_{k}(a,b)$ (see Fig.\ref{fig-1}). 
 Lemma \ref{lem-3.14} and Lemma \ref{lem-two} indicate that in Fig.\ref{fig-3}, 
 if $B_2\neq \emptyset$, then 
 $A_3=\emptyset$ and every vertex of $B_2$ is adjacent to every vertex 
 of $A_1\cup A_2$. In other words, 
 $A_1$ and $B_2\cup \{u^*\}$ form 
 a complete bipartite, and 
 $v_1v_2\cdots v_{k-3}u_a$ is a path of length $k-3$.  
Thus  $G^*\in C_k(a,b)$ for some $a\ge 1$ and 
$b=|B_2|+1\ge 1$. In what follows, we consider the spectral radius of $C_k(a,b)$.

\begin{lemma}\label{lem-3.15}
If $a,b\ge 2$, then $\rho (C_k(a,b))\le\sqrt{ab+a+b}$ unless $a=2$ and $b=2$.
\end{lemma}
\begin{proof}
Without loss of generality, assume $a\ge b\ge 2$. According to Lemma \ref{sub-inn}, we have $\rho(C_k(a,b))\le\rho(C_{7}(a,b))<\rho(C_5(a,b))$. By the knowledge of equitable partition \cite[Page 198]{Godsil}, $\rho(C_7(a,b))$ and $\rho(C_5(a,b))$ are respectively the largest roots of $f(x)$ and $g(x)$, where
\[\left\{\begin{array}{l}
f(x)= x^7-(ab+a+b+4)x^5+(5ab+3a+3b+3)x^3-(5ab+a+b)x-2ab, \\
g(x)=x^5-(ab+a+b+2)x^3+(3ab+a+b)x-2ab.
\end{array}\right.\]
Note that the derivative function of $g(x)$ is $g'(x)=5x^4-3(ab+a+b+2)x^2+3ab+a+b$. It is clear that  $g'(x)>0$ whenever $x\ge\sqrt{ab+a+b}$. We have $g(x)$ is increase whenever $x\ge\sqrt{ab+a+b}$, and thus $g(x)\ge g(\sqrt{ab+a+b})$ whenever $x\ge\sqrt{ab+a+b}$. By immediate calculation, we have
\[g(\sqrt{ab+a+b})=\sqrt{ab+a+b}\cdot (ab-a-b)-2ab.\]

If $b\ge 4$, then $g(\sqrt{ab+a+b})\ge b(4a-2a)-2ab=0$. It means $g(x)\ge0$ whenever $x\ge\sqrt{ab+a+b}$, and thus $\rho(C_5(a,b))\le \sqrt{ab+a+b}$. Hence $\rho(C_k(a,b))<\sqrt{ab+a+b}$.

If $b=3$, then $g(\sqrt{ab+a+b})=\sqrt{4a+3}\cdot(2a-3)-6a$. If $a\ge 5$, then 
\begin{align*}
    g(\sqrt{ab+a+b})&=\sqrt{4a+3}(2a-3)-6a\\
    &\ge \sqrt{23}\cdot(2a-3)-6a=(2\sqrt{23}-6)a-3\sqrt{23}\\
    &\ge (2\sqrt{23}-6)\cdot 5-3\sqrt{23}>0.
\end{align*}
It means $\rho(C_{k}(a,b))<\rho(C_5(a,b))<\sqrt{ab+a+b}$. If $a=4$, then $f(x)=x^7 - 23x^5 + 84x^3 - 67x - 24$, whose largest root is $4.327<\sqrt{ab+a+b}=\sqrt{19}$. Hence $\rho(C_k(a,b))\le \rho(C_7(a,b))<\sqrt{ab+a+b}$. If $a=3$, then $f(x)=x^7-19x^5+66x^3-51x-18$, whose largest root is $3.846<\sqrt{ab+a+b}=\sqrt{15}$. Hence $\rho(C_{k}(a,b))\le \rho(C_7(a,b))<\sqrt{ab+a+b}$.

If $b=2$, then $g(\sqrt{ab+a+b})=\sqrt{3a+2}\cdot (a-2)-4a$. If $a\ge 9$, then
\begin{align*}
    g(\sqrt{ab+a+b})&=\sqrt{3a+2}\cdot (a-2)-4a\\
    &\ge \sqrt{29}\cdot(a-2)-4a=(\sqrt{29}-4)a-2\sqrt{29}\\
    &\ge (\sqrt{29}-4)\cdot 9-2\sqrt{29}>0.
\end{align*}
It means $\rho(C_k(a,b))<\rho(C_5(a,b))<\sqrt{ab+a+b}$. If $3\le a\le 8$, as similar to the case of $b=3$, the largest root of $f(x)$ is less than $\sqrt{ab+a+b}$ (see Tab.\ref{tab-1}), and thus $\rho(C_k(a,b))\le\rho(C_7(a,b))<\sqrt{ab+a+b}$.

The proof is completed.
\end{proof}
\begin{table}[htbp]
    \centering
    \begin{tabular}{cccc}
    \hline 
         $(b,a)$& $f(x)$ & largest root &$\sqrt{ab+a+b}$ \\
         \hline
         $(3,4)$&$x^{7}-23x^{5}+84x^{3}-67x-24$& $4.3266$ & $4.3589$ \\ 
          $(3,3)$& $x^{7}-19x^{5}+66x^{3}-51x-18$ & $3.8461$ & $3.8730$\\
          $(2,8)$&$x^{7}-30x^{5}+113x^{3}-90x-32$ & $5.0753$& $5.0990$ \\
          $(2,7)$& $x^{7}-27x^{5}+100x^{3}-79x-28$ & $4.7720$ & $4.7958$ \\
          $(2,6)$ & $x^{7}-24x^{5}+87x^{3}-68x-24$ & $4.4488$& $4.4721$ \\
          $(2,5)$& $x^{7}-21x^{5}+74x^{3}-57x-20$ & $4.1001$& $4.1231$ \\
          $(2,4)$&$x^{7}-18x^{5}+61x^{3}-46x-16$ &$3.7230$ & $3.7417$ \\
          $(2,3)$& $x^{7}-15x^{5}+48x^{3}-35x-12$&$3.3065$ & $3.3166$ \\
         \hline
    \end{tabular}
    \caption{The function $f(x)$ and its largest root used in the proof of Lemma \ref{lem-3.15}.}
    \label{tab-1}
\end{table}

By immediate calculations, we have $\rho(C_5(2,2))=2.9032>7/\sqrt{6}=\frac{m(C_5(2,2))-5+2}{\sqrt{m(C_5(2,2))-5+1}}$. For $k\ge7$, we get the following result.

\begin{lemma}
$\rho(C_{k}(2,2))<\frac{7}{\sqrt{6}}=\frac{m(C_k(2,2))-k+2}{\sqrt{m(C_k(2,2))-k+1}}$ for $k\ge 7$.
\end{lemma}
\begin{proof}
By immediate calculations, we have $\rho(C_k(2,2))\le \rho(C_7(2,2))=2.84<\frac{7}{\sqrt{6}}$.
\end{proof}

\begin{proof}[\bf Proof of Theorem \ref{thm-1.1}]
According to Lemma \ref{lem-3.8}, $G^*$ is of the form (3) shown in Fig.\ref{fig-3}. If $B_2=\emptyset$, then Lemma \ref{lem-3.8'} implies that $G^*\cong SK_{k,m}$. If $B_2\ne\emptyset$, then Lemmas \ref{lem-3.14} and 
\ref{lem-two}
mean that $G^*=C_k(a,b)$ for some $a,b$. Without loss of generality, assume that $a\ge b$. It only need to show $b=1$. Suppose to the contrary that $b\ge 2$. Lemma \ref{lem-3.15} indicates that $a=b=2$. However, in this case $m=8+k-3$ is even, a contradiction.
\end{proof}

\section{Open problems for cycles}

Theorem \ref{thm-1.1} determines the extremal graphs 
for non-bipartite graphs without any short odd cycle of 
$\{ C_3,C_5,\ldots ,C_{2k+1}\}$. 
In this section, we will conclude some recent development 
on this topic and propose some 
open problems for readers. 
Let $K_k\vee I_t$ be the graph consisting of a clique on $k$ vertices
and an independent set on $t$ vertices in which each vertex of the clique is adjacent to each vertex of the independent set. 
A well-known conjecture in extremal spectral graph theory involving 
consecutive cycles states that  

\begin{conjecture}[Zhai--Lin--Shu \cite{zhai-lin-shu}] \label{conj-ZLS-1}
Let $k$ be fixed and $m$ be large enough. 
If $G$ is a graph with $m$ edges and 
\[  \lambda (G)\ge \frac{k-1 +\sqrt{4m -k^2+1}}{2}, \]
then $G$ contains a cycle of length $t$ for every $ t\le 2k+2$, unless 
$G=K_k \vee I_{\frac{1}{k}\left(m-{k \choose 2} \right)}$. 
\end{conjecture}

One may  consider  naturally the problem 
of finding the maximum spectral radius among all 
$\{C_3,C_4,\ldots ,C_{2k+2}\}$-free graphs with $m$ edges. 
We remark here that this problem is an easy consequence of 
Nikiforov in \cite{Niki2009laa}, 
which implies that the star graph on $m$ edges attains the maximum spectral radius. 
Indeed, Nikiforov \cite{Niki2009laa} 
proved that for $m\ge 10$, every $C_4$-free graph $G$ satisfies 
$\rho (G)\le \sqrt{m}$, equality holds if and only if 
$G$ is uniquely a star with $m$ edges. 
The above problem is a direct corollary by 
noting that the star graph contains no copies of 
$C_t$ for every $3\le t\le 2k+2$.

The following Conjecture  \ref{conj-ZLS} seems weaker than 
Conjecture \ref{conj-ZLS-1} at first glance. 
While they are equivalent since the bound in right hand side is monotonically 
increasing on $k\in [2,+\infty )$. 
So it is reasonable to attribute this conjecture to Zhai, Lin and Shu. 

\begin{conjecture}[Zhai--Lin--Shu] \label{conj-ZLS}
Let $k$  be fixed and $G$ be a graph of sufficiently large size $m$ 
without isolated vertices. 
If $G$ is  $C_{2k+1}$-free or $C_{2k+2}$-free, then 
\[  \lambda (G)\le \frac{k-1 +\sqrt{4m -k^2+1}}{2}, \]
equality holds if and only if  
$G=K_k \vee I_{\frac{1}{k}\left(m-{k \choose 2} \right)}$. 
\end{conjecture}

In 2021,  Zhai, Lin and Shu \cite{zhai-lin-shu} 
proved this conjecture in the case of $k=2$ and odd $m$, 
and later Gao, Lou and Huang \cite{MLH2022} proved the case of 
$k=2$ and even $m$. 
These cases were also  provided by Li, Sun and Wei \cite{LSW2022}.  
Conjecture \ref{conj-ZLS} remains open for $k\ge 3$. 

Let $C_t^{\triangle}$ denote the graph on $t+1$ vertices obtained from 
$C_t$ and $C_3$ by identifying an edge. 
It was proved in \cite{zhai-lin-shu} that the complete bipartite graphs attain the 
maximum spectral radius among $\{C_3^{\triangle}, C_4^{\triangle}\}$-free graphs with 
$m$ edges. 
In \cite{Nikiforov}, Nikiforov enhanced  that the same result 
still holds for $C_3^{\triangle}$-free graphs. 
Very recently, 
 Li, Sun and Wei \cite{LSW2022} determined 
 the extremal graph for $C_4^{\triangle}$-free or 
 $C_5^{\triangle}$-free when the size $m$ is odd, 
 and soon after,  Fang, You and Huang \cite{FYH2022} determined 
 the extremal graph for even $m$. 
Observe that $C_{2k+1}\subseteq C_{2k}^{\triangle}$ and 
$C_{2k+2}\subseteq C_{2k+1}^{\triangle}$. 
Motivated by Conjecture \ref{conj-ZLS}, 
Yongtao Li  tells privately us that it is also interesting to consider the following conjecture. 

\begin{conjecture}[Yongtao Li]
Let $k\ge 3$ and $G$ be a graph of sufficiently large size $m$ 
without isolated vertices. 
If $G$ is  $C_{2k}^{\triangle}$-free or $C_{2k+1}^{\triangle}$-free, then 
\[  \lambda (G)\le \frac{k-1 +\sqrt{4m -k^2+1}}{2}, \]
equality holds if and only if  
$G=K_k \vee I_{\frac{1}{k}\left(m-{k \choose 2} \right)}$. 
\end{conjecture}

\section*{Acknowledgments}
This work is supported by NSFC (Nos. 12001544, 12061074, 11971274, 11901540), Natural Science Foundation of Hunan Province (No. 2021JJ40707) and the China
Postdoctoral Science Foundation (No. 2019M661398). We are so grateful to Dr. Yongtao Li, who provides many valuable suggestions.

\end{document}